\documentclass[12pt]{article} 
\usepackage[margin=1in]{geometry}
\usepackage{setspace}
\usepackage{graphicx}	
\usepackage{subcaption}
\usepackage{float} 
\usepackage[toc]{appendix}
\usepackage[draft]{todonotes} 
\usepackage{enumerate}	
\usepackage{paralist}	
\usepackage{amsmath, amsthm, amssymb,mathtools}
\usepackage{thm-restate}

\usepackage[pdftex,hypertexnames=false,linktocpage=true,hidelinks]{hyperref}

\usepackage{fullpage}
\usepackage{amsmath}
\usepackage{amsfonts}

\newtheorem{theorem}{Theorem}[section]

\newtheorem{const}[theorem]{Construction}

\newtheorem{claim}{Claim}

\numberwithin{equation}{section}

\definecolor{amber}{rgb}{1.0, 0.75, 0.0}

\definecolor{forest}{rgb}{0.0, 0.5, 0.0}

\definecolor{cadmium}{rgb}{0.93, 0.53, 0.18}

\definecolor{byzantine}{rgb}{0.74, 0.2, 0.64}

\definecolor{brilliantrose}{rgb}{1.0, 0.33, 0.64}

\definecolor{caribbeangreen}{rgb}{0.0, 0.8, 0.6}

\definecolor{electriccyan}{rgb}{0.0, 1.0, 1.0}

\definecolor{periwinkle}{rgb}{0.8, 0.8, 1.0}

\definecolor{steelblue}{rgb}{0.27, 0.51, 0.71}

\usetikzlibrary{decorations.pathmorphing}

\tikzset{snake it/.style={decorate, decoration=snake}}

\title{A note on the rainbow Tur\'an number of brooms with length 2 handles}
\author{Anastasia Halfpap\footnote{Mathematics Department, Truman State University, Kirksville, MO, USA.  E-mail: \texttt{ahalfpap@truman.edu}. } }

\begin{document}

\maketitle

\begin{abstract}
    For a fixed graph $F$, the rainbow Tur\'an number $\mathrm{ex^*}(n,F)$ is the largest number of edges possible in an $n$-vertex graph which admits a rainbow-$F$-free proper edge-coloring. We focus on the rainbow Tur\'an numbers of trees obtained by appending some number of pendant edges to one end of a length 2 path; we call such a tree with $k$ total edges a $k$-edge \textit{broom} with length $2$ handle, denoted by $B_{k,2}$. Study of $\mathrm{ex^*}(n,B_{k,2})$ was initiated by Johnston and Rombach, who claimed a proof asymptotically establishing the value of $\mathrm{ex^*}(n,B_{k,2})$ for all $k$. We correct an error in this original argument, identifying two small cases in which the value claimed in the literature is incorrect; in all other cases, we recover the originally claimed value. Our argument also characterizes the extremal constructions for $\mathrm{ex^*}(n,B_{k,2})$ for certain congruence classes of $n$ modulo $k$.
\end{abstract}

\section{Introduction}
Given two graphs $G$ and $F$, we say that $G$ is \textit{$F$-free} if $G$ contains no subgraph isomorphic to $F$. The \textit{Tur\'an number of $F$}, denoted $\mathrm{ex}(n,F)$, is the maximum number of edges in an $n$-vertex $F$-free graph. Study of $\mathrm{ex}(n,F)$ forms a cornerstone of extremal combinatorics, and has inspired a variety of generalizations and extensions in recent years. In particular, an \textit{edge-coloring} of a graph $G$ is a function $c: E(G) \rightarrow \mathbb{N}$. For $e \in E(G)$, we say that $c(e)$ is the \textit{color} of $e$. We say that an edge-coloring $c$ of $G$ is \textit{proper} if no two incident edges of $G$ receive the same color under $c$. We say that $c$ is a \textit{rainbow edge-coloring} (or that $G$ is \textit{rainbow under $c$}) if all edges of $G$ receive distinct colors under $c$ (that is, if $c$ is injective). Given graphs $G$ and $F$ and an edge-coloring $c$ of $G$, we say that $G$ is \textit{rainbow-$F$-free} (under $c$) if $G$ does not contain any subgraph which is both isomorphic to $F$ and rainbow under $c$.
The \textit{rainbow Tur\'an number} of $F$, denoted $\mathrm{ex^*}(n,F)$, is the largest number of edges in an $n$-vertex graph $G$ which admits a proper edge-coloring under which $G$ is rainbow-$F$-free.

Study of $\mathrm{ex^*}(n,F)$ was initiated by Keevash, Mubayi, Sudakov, and Verstra\"ete~\cite{KMSV}, who showed (among other things) that 
\[\mathrm{ex}(n,F) \leq \mathrm{ex^*}(n,F) \leq \mathrm{ex}(n,F) + o(n^2)\]
for all graphs $F$. We have that $\mathrm{ex}(n,F) = \Theta(n^2)$ for all non-bipartite $F$ (and in fact the value of $\mathrm{ex}(n,F)$ is asymptotically understood for non-bipartite $F$; see \cite{ErSt},\cite{ErSi}). Thus, the study of $\mathrm{ex^*}(n,F)$ has primarily focused on cases where $F$ is bipartite. In this regime, it is natural to begin with trees and forests, and substantial work has been devoted to finding $\mathrm{ex^*}(n,F)$ for various trees (see, for instance, \cite{bednar2022rainbow}, \cite{ErGyMe}, \cite{halfpaprainbowp5}, \cite{JoPaSa}, \cite{JoRo}). 

In this note, we denote by $P_{\ell}$ the path on $\ell$ \textit{edges} -- that is, the path of length $\ell$. A \textit{broom} is a tree obtained from some path $P_{\ell}$ by appending some number of pendant edges to one endpoint of $P_{\ell}$. We call the path $P_{\ell}$ the \textit{handle} of this broom and the pendant edges its \textit{bristles}. Denote by $B_{k,\ell}$ the broom with a length $\ell$ handle and $k$ total edges (so, $B_{k,\ell}$ has $k - \ell$ bristles). Brooms of various dimensions are depicted in Figure~\ref{brooms}.

\begin{figure}[h!]
    \centering

\begin{tikzpicture}[scale = 1]
\draw (-1,0) coordinate (a0);
\draw (0,0) coordinate (a1);
\draw (1,0) coordinate (a2);
\draw (2,0) coordinate (a3);
\draw (2,1) coordinate (a4);
\draw (2,-1) coordinate (a5);
\draw (2,0.5) coordinate (a6);
\draw (2,-0.5) coordinate (a7);

\draw (4,0) coordinate (b1);
\draw (5,0) coordinate (b2);
\draw (6,0) coordinate (b3);
\draw (7,0) coordinate (b4);
\draw (8,1) coordinate (b5);
\draw (8,-1) coordinate (b6);

\filldraw (a0) circle (0.05 cm);
\filldraw (a1) circle (0.05 cm);
\filldraw (a2) circle (0.05 cm);
\filldraw (a3) circle (0.05 cm);
\filldraw (a4) circle (0.05 cm);
\filldraw (a5) circle (0.05 cm);
\filldraw (a6) circle (0.05 cm);
\filldraw (a7) circle (0.05 cm);

\filldraw (b1) circle (0.05 cm);
\filldraw (b2) circle (0.05 cm);
\filldraw (b3) circle (0.05 cm);
\filldraw (b4) circle (0.05 cm);
\filldraw (b5) circle (0.05 cm);
\filldraw (b6) circle (0.05 cm);

\draw (a0) -- (a3);
\draw (a2) -- (a4);
\draw (a2) -- (a5);
\draw (a2) -- (a6);
\draw (a2) -- (a7);

\draw (b1) -- (b4);
\draw (b4) -- (b5);
\draw (b4) -- (b6);
\end{tikzpicture}

    \caption{The brooms $B_{7,2}$ (left) and $B_{5,3}$ (right).}
    \label{brooms}
\end{figure}
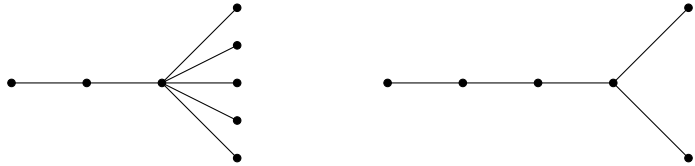

In general, determination of $\mathrm{ex^*}(n,B_{k,\ell})$ seems very difficult. The family of all brooms contains all paths, and we only understand $\mathrm{ex^*}(n,P_{\ell})$ precisely for $\ell \leq 5$ (see~\cite{JoRo} and~\cite{halfpaprainbowp5}). However, for small values of $\ell$, finding $\mathrm{ex^*}(n, B_{k,\ell})$ seems much more tractable than other rainbow Tur\'an problems for trees. Systematic study of $\mathrm{ex^*}(n,B_{k,\ell})$ was initiated by Johnston and Rombach, who claim the following result (Lemma 4.1 in~\cite{JoRo}).

\begin{theorem}[Johnston-Rombach~\cite{JoRo}]\label{JR theorem}
\[
\mathrm{ex^*}(n,B_{k,2}) = 
 \begin{cases}\frac{k}{2}n + O(1) & \text{ if } k \text{ is odd} \\
\frac{k^2}{2(k+1)}n + O(1) & \text{ if } k \text{ is even}
\end{cases}
\]
\end{theorem}

Unfortunately, the statement of Theorem~\ref{JR theorem} is incorrect for two small values of $k$, and the original proof given by Johnston and Rombach contains an error in all cases where $k$ is even. The purpose of this note is to amend the statement of Theorem~\ref{JR theorem} to capture these small exceptions and to provide a proof which both captures these exceptional cases and addresses the broader issue with the original argument. We state a corrected version of the theorem now.

\begin{restatable}{theorem}{correctedtheorem}\label{corrected theorem}
Let $k \geq 2$.
\[
\mathrm{ex^*}(n,B_{k,2}) = 
 \begin{cases}\frac{k}{2}n + O(1) & \text{ if } k \text{ is odd} \\
 \frac{k-1}{2}n + O(1) & \text{ if } k \in\{2,4\}\\
\frac{k^2}{2(k+1)}n + O(1) & \text{ if } k \text{ is even and } k \geq 6
\end{cases}
\]
\end{restatable}

The remainder of this note is organized as follows. In Subsection~\ref{section: notation}, we briefly summarize the notation used in our arguments. In Section~\ref{section: proof}, we discuss the error in the original proof of Theorem~\ref{JR theorem} and we give a proof of Theorem~\ref{corrected theorem}.

\subsection{Notation}\label{section: notation}

We endeavor to use notation standard to the area. For reference, we collect a variety of ubiquitous definitions here. 

A \textit{graph} $G$ is defined by a pair of sets: a vertex set $V(G)$ and an edge set $E(G)$ of unordered pairs from $V$. All graphs in this note are finite and simple. Vertices $u,v$ are \textit{adjacent} in $G$ if $uv \in E(G)$ and vertex $v$ is \textit{incident} to edge $e$ if $v \in e$. The \textit{degree} of a vertex $v \in V(G)$, denoted $d(v)$, is the number of edges in $E(G)$ which are incident to $v$. We denote the maximum degree in $G$ by $\Delta(G)$. The \textit{neighborhood} of $v$, denoted $N(v)$, is the set of vertices adjacent to $v$.

In addition to the edge-coloring definitions given in the introduction, we use several standard terms to describe edge-colored graphs. Given a graph $G$, an edge-coloring $c$ is a $k$-edge-coloring if the image of $c$ has size $k$; we say that $G$ is $k$-edge-colored under $c$. Similarly, $G$ is $k$\textit{-edge-colorable} if there exists a proper coloring $c$ of $G$ using $k$ colors. When it will not cause confusion, we typically omit the explicit reference to the coloring function $c$ -- so, say that $G$ is a $k$-edge-colored graph. In an edge-colored graph $G$, a \textit{color class} is a set of edges which all receive the same edge-color. 

We will use some standard graph operations to describe constructions. Given a graph $G$, the \textit{complement} of $G$, denoted $\overline{G}$, is the graph with $V(\overline{G}) = V(G)$ and $e \in E(\overline{G})$ if and only if $e \not\in E(G)$. Given two graphs $G$ and $H$ on (disjoint) vertex sets $V(G)$ and $V(H)$, the \textit{union} $G \cup H$ is the graph on vertex set $V(G) \cup V(H)$ and edge set $E(G) \cup E(H)$.

We recall that in this note, we denote by $P_{\ell}$ the path on $\ell$ \textit{edges}. We denote specific paths by the concatenation of their vertices, e.g., the path $xyz$ is a copy of $P_2$ with edges $xy$ and $yz$.

\section{Proof of Theorem~\ref{corrected theorem}}\label{section: proof}

We begin by briefly discussing $\mathrm{ex^*}(n,B_{2,2})$ and $\mathrm{ex^*}(n,B_{4,2})$. The graphs $B_{2,2}$ and $B_{4,2}$ are depicted in Figure~\ref{22 and 42}. 
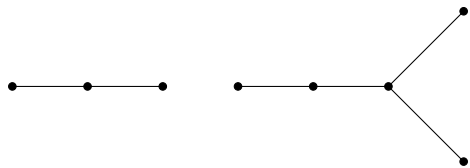
\begin{figure}[h!]
    \centering

\begin{tikzpicture}[scale = 1]

\draw (0,0) coordinate (a1);
\draw (1,0) coordinate (a2);
\draw (2,0) coordinate (a3);

\draw (3,0) coordinate (b2);
\draw (4,0) coordinate (b3);
\draw (5,0) coordinate (b4);
\draw (6,1) coordinate (b5);
\draw (6,-1) coordinate (b6);

\filldraw (a1) circle (0.05 cm);
\filldraw (a2) circle (0.05 cm);
\filldraw (a3) circle (0.05 cm);

\filldraw (b2) circle (0.05 cm);
\filldraw (b3) circle (0.05 cm);
\filldraw (b4) circle (0.05 cm);
\filldraw (b5) circle (0.05 cm);
\filldraw (b6) circle (0.05 cm);

\draw (a1) -- (a3);

\draw (b2) -- (b4);
\draw (b4) -- (b5);
\draw (b4) -- (b6);
\end{tikzpicture}

    \caption{The brooms $B_{2,2}$ (left) and $B_{4,2}$ (right).}
    \label{22 and 42}
\end{figure}
While $B_{2,2}$ is a degenerate case, we include it for completeness.

Observe that any properly edge-colored copy of $B_{2,2}$ is rainbow. Since a $B_{2,2}$-free graph has maximum degree 1, it is immediate that 
\[\mathrm{ex^*}(n,B_{2,2}) = \mathrm{ex}(n, B_{2,2}) = \left\lfloor \frac{n}{2}\right\rfloor = \frac{1}{2}n + O(1).\]
Again, this is a somewhat degenerate case. The case of $B_{4,2}$ is more interesting. It is possible to show by direct case analysis that $\mathrm{ex^*}(n,B_{4,2}) = \frac{3}{2}n+O(1)$. However, as the argument for Theorem~\ref{corrected theorem} will yield the value of $\mathrm{ex^*}(n,B_{4,2})$ without case analysis specific to $B_{4,2}$, we wait until the proof of Theorem~\ref{corrected theorem} to establish this value.

Before we prove Theorem~\ref{corrected theorem}, we discuss the reason why these exceptions have been overlooked. In the proof of Theorem~\ref{JR theorem}, the following construction is used to derive a lower bound on $\mathrm{ex^*}(n,B_{k,2})$ for even $k$. 

\begin{const}\label{JR construction}
Fix $n\geq 1$ and $k$ an even integer. Take a properly $(k+1)$-edge-colored copy of $K_{k+1}$, and delete a color class to obtain a $k$-edge-colored subgraph with $\frac{(k+1)k}{2} - \frac{k}{2} = \frac{k^2}{2}$ edges. Taking $\left\lfloor \frac{n}{k+1} \right\rfloor$ disjoint copies of this graph yields a construction with $\frac{k^2}{2(k+1)}n + O(1)$ edges. 
\end{const}

However, it is not demonstrated that Construction~\ref{JR construction} is actually rainbow-$B_{k,2}$-free. Johnston and Rombach argue that in a rainbow-$B_{k,2}$-free graph $G$, if a component $H$ has maximum degree $ \geq k$, then $H$ must have exactly $k+1$ vertices and receive exactly $k$ edge-colors. As $k+1$ is odd, these conditions imply that each color class has size at most $\frac{k}{2}$, so that $H$ contains at most $\frac{k^2}{2}$ edges. Counting edges in $G$ component by component, these necessary conditions imply that $\mathrm{ex^*}(n,B_{k,2}) \leq \frac{k^2}{2(k+1)}n + O(1)$. Construction~\ref{JR construction} satisfies all of these necessary conditions, but it is not established that these necessary conditions are in fact sufficient to guarantee that a graph is rainbow-$B_{k,2}$-free. As we shall see in the proof of Theorem~\ref{corrected theorem}, Construction~\ref{JR construction} is actually never rainbow-$B_{k,2}$-free. In Figure~\ref{B42}, we illustrate Construction~\ref{JR construction} in the case $k = 4$, and explicitly display a copy of $B_{4,2}$ contained within the construction.

\begin{figure}[h!]
    \centering

\begin{tikzpicture}[scale = 1.5]
\draw ({0},{1}) coordinate (a1);
\draw ({0.951},{0.309}) coordinate (a2);
\draw ({-0.951},{0.309}) coordinate (a5);
\draw ({0.587},{-0.809}) coordinate (a3);
\draw ({-0.587},{-0.809}) coordinate (a4);

\draw ({0+3},{1}) coordinate (b1);
\draw ({0.951 + 3},{0.309}) coordinate (b2);
\draw ({-0.951 + 3},{0.309}) coordinate (b5);
\draw ({0.587 + 3},{-0.809}) coordinate (b3);
\draw ({-0.587 + 3},{-0.809}) coordinate (b4);

\draw[thick, red] (a2) -- (a5);
\draw[thick, red] (a3) -- (a4);
\draw[thick, dashed, cadmium] (a1) -- (a3);
\draw[thick, dashed, cadmium] (a4) -- (a5);
\draw[double, forest] (a1) -- (a5);
\draw[double, forest] (a2) -- (a4);
\draw[snake it, blue] (a3) -- (a5);
\draw[snake it, blue] (a1) -- (a2);

\draw[snake it, blue] (b1) -- (b2);
\draw[double, forest] (b2) -- (b4);
\draw[thick, red] (b3) -- (b4);
\draw[thick, dashed, cadmium] (b4) -- (b5);

\filldraw (a1) circle (0.05 cm) node[above]{$x$};
\filldraw (a2) circle (0.05 cm) node[right]{$y$};
\filldraw (a3) circle (0.05 cm) node[right]{$z$};
\filldraw (a4) circle (0.05 cm) node[left]{$u$};
\filldraw (a5) circle (0.05 cm) node[left]{$v$};

\filldraw (b1) circle (0.05 cm) node[above]{$x$};
\filldraw (b2) circle (0.05 cm) node[right]{$y$};
\filldraw (b3) circle (0.05 cm) node[right]{$z$};
\filldraw (b4) circle (0.05 cm) node[left]{$u$};
\filldraw (b5) circle (0.05 cm) node[left]{$v$};

\end{tikzpicture}

    \caption{Left, a $4$-edge-colored subgraph of $K_5$ obtained by removing one color class from a $5$-edge-coloring. Right, one of the rainbow-$B_{4,2}$-copies contained in this subgraph.}
    \label{B42}
\end{figure}
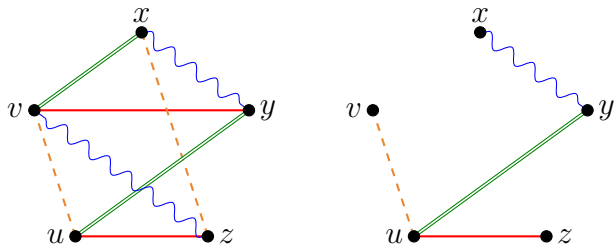

As we shall see in the proof of Theorem~\ref{corrected theorem}, the failure of Construction~\ref{JR construction} is not specific to the case $k = 4$; for any even $k$, Construction~\ref{JR construction} contains rainbow copies of $B_{k,2}$. However, the failure of Construction~\ref{JR construction} does not itself imply that the conclusion of Theorem~\ref{JR theorem} is incorrect for even $k$. Indeed, in almost all cases, we are able to recover the bound in Theorem~\ref{JR theorem} via somewhat different constructions (which also satisfy the necessary conditions given by Johnston and Rombach). To quickly argue that our constructions are $k$-edge-colorable, the following theorem of Plantholt will be useful.

\begin{theorem}[Plantholt~\cite{Plantholt}]
Let $k > 0$ be an even integer and $G$ be a graph with $|V(G)| = k+1$ and $\Delta(G) = k$. Then $G$ is $k$-edge-colorable if and only if $|E(G)| \leq \frac{k^2}{2}$.
\end{theorem}
With Plantholt's Theorem in hand, we are ready to prove Theorem~\ref{corrected theorem}, which we re-state here for convenience. In the interest of self-containment, prove the full statement of Theorem~\ref{corrected theorem}, but we note that for $k$ odd, our proof is identical to the original argument of Johnston and Rombach; a new argument is required only for $k$ even.

\correctedtheorem*

\begin{proof}
As noted above, we immediately have $\mathrm{ex^*}(n,B_{2,2}) = \frac{1}{2}n + O(1)$ and this is a somewhat degenerate case. For the remainder of the proof, we assume $k \geq 3$. We begin by making some general observations that lead to an upper bound for all $k \geq 3$. Suppose that $H$ is a connected graph which admits a rainbow-$B_{k,2}$-free proper edge-coloring and which contains a vertex $v$ with $d(v) \geq k$. It is straightforward to verify the following facts. Firstly, if any neighbor $u$ of $v$ is adjacent to any non-neighbor $w$ of $v$, then the path $wuv$ forms the handle of a rainbow $B_{k,2}$-copy (whose $k-2$ bristles are selected from the at least $k-1$ edges which are incident to $v$ and not equal to $vu$). Secondly, if $d(v) > k$ and there exist $u,w \in N(v)$ with $uw \in E(H)$, then the path $wuv$ forms the handle of a rainbow $B_{k,2}$-copy (whose $k-2$ bristles are selected from the at least $k-1$ edges which are incident to $v$ and not equal to either $vu$ or $vw$). Thus, if $\Delta(H) > k$, then $H$ is a star, and if $\Delta(H) = k$, then $H$ is a subgraph of $K_{k+1}$.  We conclude that in an $n$-vertex rainbow-$B_{k,2}$-free graph $G$, each component is either a star or has maximum degree $\leq k$, so $|E(G)| \leq \frac{k}{2}n$. Thus, $\mathrm{ex^*}(n,B_{k,2}) \leq \frac{k}{2}n$ for all $k \geq 3$. 

We show that this upper bound is tight when $k$ is odd. When $k$ is odd, it is well-known that $K_{k+1}$ admits a proper $k$-edge-coloring (see~\cite{Kirkman}). We claim that any proper $k$-edge-coloring of $K_{k+1}$ is rainbow-$B_{k,2}$-free. Indeed, suppose that $c$ is a proper $k$-edge-coloring of $K_{k+1}$ (say assigning colors $1, 2, \dots k$ to the edges of $K_{k+1}$) and consider some $B_{k,2}$-copy $B$ in $K_{k+1}$. Let $v$ be the unique vertex of degree $k-1$ in $B$, and let $w$ be the unique vertex which is not adjacent to $v$ in $B$. Note that since $c$ is proper, the edges incident to $v$ in $B_{k,2}$ receive $k-1$ distinct edge-colors; without loss of generality, in $B$, $v$ is incident to edges of colors $1, 2, \dots, k-1$. Since $c$ is a $k$-edge-coloring, each color class of $c$ is a perfect matching of $K_{k+1}$, so $v$ is incident to an edge of color $k$ in $K_{k+1}$, which must be $vw$. However, since $c$ is proper and the edge incident to $w$ in $B$ is not $vw$, this means that $B$ does not contain an edge of color $k$. As every edge-color in $B$ is from $\{1,2, \dots, k\}$, it follows that $B$ receives exactly $k-1$ edge-colors under $c$, so is not rainbow. If $n$ is divisible by $k$, we can thus construct an $n$-vertex, rainbow-$B_{k,2}$-free graph with $\frac{k}{2}n$ edges by taking $\frac{n}{k+1}$ disjoint, $k$-edge-colored copies of $K_{k+1}$. In general, taking $\lfloor \frac{n}{k+1} \rfloor$ copies of $K_{k+1}$ shows that $\mathrm{ex^*}(n,B_{k,2}) \geq \frac{k}{2} n + O(1)$ when $k$ is odd.

We next consider the cases when $k$ is even (and larger than $2$). We certainly have $\frac{k-1}{2}n + O(1) \leq \mathrm{ex^*}(n,B_{k,2})$ for all even $k$, as it suffices to find an $n$-vertex (nearly) $(k-1)$-regular graph which is $(k-1)$-edge-colorable. Such graphs certainly exist for all $n$: take, for instance, $\lfloor \frac{n}{k} \rfloor$ disjoint copies of $K_k$ along with $n - k\lfloor \frac{n}{t}\rfloor$ isolated vertices. 
We now work to determine when a $(k-1)$-regular construction is (and is not) optimal. Suppose that there exists a rainbow-$B_{k,2}$-free graph $G$ on some number of vertices $n$ with $|E(G)| > \frac{k-1}{2}n$. Consider a densest component $H$ of $G$. We have $|E(H)| > \frac{k-1}{2}|V(H)|$, so $H$ contains a vertex of degree at least $k$, say $v$. As noted above, $H$ is either a star or $d(v) = k$ and $H$ is a subgraph of $K_{k+1}$. Since a star clearly does not have the desired density (since $k > 2$), $H$ must be a subgraph of $K_{k+1}$. Note also that $H$ must be $k$-edge-colored to avoid a rainbow $B_{k,2}$-copy; indeed, if $v$ is incident to edges of colors $1,2, \dots,k$ and there are $u,w \in N(v)$ such that $uw$ is an edge with $c(uw) = k+1$, then the path $wuv$ forms the handle of a rainbow-$B_{k,2}$-copy (whose $k-2$ bristles are exactly the $k-2$ edges which are incident to $v$ and not equal to either $vu$ or $vw$). 

As $k+1$ is odd, each color class in $H$ has size at most $\frac{k}{2}$. Adding sizes of color classes, it is clear that $|E(H)| \leq \frac{k^2}{2}$. On the other hand, our density assumption on $H$ says that
\[E(H) > \frac{k-1}{2}(k+1) = \frac{k^2 - 1}{2}.\]
Thus, $H$ is a rainbow-$B_{k,2}$-free, $k$-edge-colored subgraph of $K_{k+1}$ containing exactly $\frac{k^2}{2}$ edges. As $H$ is a densest component of $G$, we have $|E(G)| \leq \frac{k^2}{2(k+1)}n$.

Every subgraph of $K_{k+1}$ with exactly $\frac{k^2}{2}$ edges has maximum degree $k$ and is thus $k$-edge-colorable by Plantholt's Theorem. However, it is not clear that every such subgraph has a rainbow-$B_{k,2}$-free $k$-edge-coloring. Actually, such a subgraph may or may not be rainbow-$B_{k,2}$-free. In light of the above discussion, if there exists a subgraph $H$ of $K_{k+1}$ with exactly $\frac{k^2}{2}$ edges and which admits a rainbow-$B_{k,2}$-free proper $k$-edge-coloring, then $\mathrm{ex^*}(n,B_{k,2}) = \frac{k^2}{2(k+1)}n + O(1)$ (as we can match the upper bound by taking disjoint copies of this subgraph $H$). If no such subgraph $H$ exists, then we achieve a contradiction to our supposition that there exists an $n$-vertex rainbow-$B_{k,2}$-free graph $G$ with $|E(G)| > \frac{k-1}{2}n$, and thus we have $\mathrm{ex^*}(n,B_{k,2}) = \frac{k-1}{2}n + O(1)$.

\begin{claim}\label{no t-1}
Let $H$ be a $k$-edge-colored subgraph of $K_{k+1}$ with exactly $\frac{k^2}{2}$ edges. $H$ contains a rainbow $B_{k,2}$-copy if and only if $H$ contains a vertex of degree exactly $k-1$.
\end{claim}

\begin{proof}[Proof of Claim]

Suppose first that $u \in V(H)$ has $d(u)  = k - 1$. Thus, $u$ is incident to edges colored by exactly $k-1$ of the $k$ colors appearing in $H$. Without loss of generality, the edge-colors of $H$ are $\{1,2,\dots,k\}$ and $u$ is not incident to color $1$. Now, let $w$ be the unique vertex in $H$ such that $uw \not \in E(H)$. Note that each color class of $H$ has size exactly $\frac{k}{2}$. In particular, for each $i \in \{1,2, \dots, k\}$, there exists exactly one vertex in $H$ which is not incident to an edge of color $i$. Since $u$ is the \textit{unique} vertex in $H$ which is not incident to an edge of color $1$, there must be some neighbor $x$ of $w$ with $c(wx) = 1$. Since $d(u) = k-1$, we must have that $ux \in E(H)$. Without loss of generality, $c(ux) = 2$. Now, $u$ has exactly $k-2$ neighbors which are not in $\{w,x\}$, and is adjacent to these neighbors via edges of colors $\{3, \dots, k\}$. We conclude that $H$ contains a rainbow $B_{k,2}$-copy with handle $wxu$.

On the other hand, suppose $H$ contains a rainbow $B_{k,2}$-copy $B$, say with handle $wxu$, where $u$ is the vertex with degree $k-1$ in $B$. We claim that $u$ also has degree exactly $k-1$ in $H$ (and that $w$ is the unique vertex in $H$ not adjacent to $u$). Indeed, without loss of generality, $c(wx) = 1$ and $c(xu) = 2$. Because $B$ is a rainbow $B_{k,2}$-copy, we know that $u$ has $k-2$ neighbors not in $\{w,x\}$, and that $u$ is adjacent to all of these via edges of colors not equal to $1$ or $2$, say $\{3,4, \dots, k\}$. So $d(u) \geq k-1$. Now, note that if $uw$ is an edge, then $c(uw) \not \in \{1,2, \dots , k\}$ by the properness of the edge-coloring of $H$. So if $uw \in E(H)$, then $c(uw)$ is a new color $k+1$, contradicting the assumption that $H$ is $k$-edge-colored. We conclude that $d(u) = k-1$.
\end{proof}

Applying Claim~\ref{no t-1} and Plantholt's Theorem, it follows that $\mathrm{ex^*}(n,B_{k,2}) = \frac{k^2}{2(k+1)}n + O(1)$ if and only if there exists a subgraph of $K_{k+1}$ with exactly $\frac{k^2}{2}$ edges and no vertex of degree exactly $k-1$. Moreover, if no such subgraph of $K_{k+1}$ exists, then $\mathrm{ex^*}(n,B_{k,2}) = \frac{k-1}{2}n + O(1)$. We shall call a $\frac{k^2}{2}$-edge subgraph of $K_{k+1}$ \textit{good} if it has no degree $k-1$ vertex, and \textit{bad} if not.

\begin{claim}\label{Good subgraphs}
$K_{k+1}$ has a good subgraph for every even $k \geq 6$. $K_5$ (and $K_3$) do not have good subgraphs.
\end{claim}

\begin{proof}
It is clear that $K_3$ has no good subgraph. When $k = 4$, we have $\frac{k^2}{2} = 8$. A subgraph $H$ of $K_5$ with exactly $8$ edges either has $\overline{H} = P_1 \cup P_1$  or $\overline{H} = P_2$. In either case, $\overline{H}$ contains vertices of degree $1$, so $H$ contains vertices of degree $k-1 = 3$. Thus, no subgraph of $K_5$ is good.

On the other hand, for $k \geq 6$, the above argument indicates a method for creating a good subgraph $H$: instead, construct $\overline{H}$ so that $|E(\overline{H})| = \frac{k}{2}$ and $\overline{H}$ contains no vertex of degree $1$. This is always possible for $k \geq 6$: for instance, let $\overline{H} = C_{\frac{k}{2}} \cup E_{\frac{k}{2} + 1}$ (where $E_m$ is the empty graph on $m$ vertices). Note that we do require $k \geq 6$ for this construction, as $C_{\frac{k}{2}}$ is not a simple graph for $k \in \{2,4\}$. Every simple graph on 1 or 2 edges contains a degree 1 vertex.
\end{proof}

In light of Claim~\ref{Good subgraphs} and the above discussion, the theorem follows immediately.
\end{proof}

We note that a full characterization of extremal constructions for $\mathrm{ex^*}(n,B_{k,2})$ is difficult because of the sensitivity to divisibility conditions; for $k \not\in \{2,4\}$, if $n$ is not divisible by $k+1$, then we cannot exactly match the upper bounds of Theorem~\ref{corrected theorem} because the only possible components achieving the extremal density have exactly $k+1$ vertices. It is likely not worth the effort to catalog the extremal constructions for all congruence classes of $n$ modulo $k$. However, in the ``nicest'' divisibility cases, we can precisely describe the possible extremal constructions, which will give some insight into the dichotomy between the even and odd cases. The argument in Claim~\ref{Good subgraphs} demonstrates that there are many non-isomorphic extremal constructions for $\mathrm{ex^*}(n,B_{k,2})$ when $k$ is even. When $k \in \{2,4\}$, any $n$-vertex $(k-1)$-regular graph which admits a proper $(k-1)$-edge-coloring is extremal. When $k \geq 6$ and $k+1$ divides $n$, extremal constructions come from taking disjoint unions of good subgraphs of $K_{k+1}$. Not only are there in general multiple $2$-regular choices for $\overline{H}$ which will yield good subgraphs, but any choice for $\overline{H}$ which contains exactly $\frac{k}{2}$ edges and no degree $1$ vertices is allowed; as $k$ grows, the number of such constructions also increases. On the other hand, when $k$ is odd, the argument of Theorem~\ref{corrected theorem} demonstrates that there are \textit{not} many extremal constructions for $\mathrm{ex^*}(n,B_{k,2})$. In the case where $k+1$ divides $n$, we have $\mathrm{ex^*}(n,B_{k,2}) = \frac{k}{2}n$ exactly. Because in a rainbow-$B_{k,2}$-free graph, every component with maximum degree $\geq k$ must be either a star or a subgraph of $K_{k+1}$, it is clear that taking disjoint copies of $K_{k+1}$ is the unique way to achieve $\mathrm{ex^*}(n,B_{k,2})$ in this case.

\bibliographystyle{abbrvurl}
\bibliography{master.bib}

@article {bednar2022rainbow,
  AUTHOR = "Bednar, Vic and Bushaw, Neal",
  TITLE = "Rainbow {T}ur{\'a}n {M}ethods for {T}rees",
  JOURNAL = {Australas. J. Combin.},
  FJOURNAL = {The Australasian Journal of Combinatorics},
  VOLUME = {91},
  YEAR = {2025},
  PAGES = {266--281}
}

@article {W,
    AUTHOR = {Wood, David R.},
     TITLE = {On the number of maximal independent sets in a graph},
   JOURNAL = {Discrete Math. Theor. Comput. Sci.},
  FJOURNAL = {Discrete Mathematics \& Theoretical Computer Science. DMTCS.},
    VOLUME = {13},
      YEAR = {2011},
    NUMBER = {3},
     PAGES = {17--19},
   MRCLASS = {05C69 (05C35)},
  MRNUMBER = {2832630},
MRREVIEWER = {Eugen Mandrescu},
       DOI = {10.1137/0206036},
       URL = {https://doi.org/10.1137/0206036},
}

@article {ErSi,
    AUTHOR = {Erd\H{o}s, Paul and Simonovits, Mikl\'{o}s},
     TITLE = {A limit theorem in graph theory},
   JOURNAL = {Studia Sci. Math. Hungar.},
  FJOURNAL = {Studia Scientiarum Mathematicarum Hungarica. Combinatorics,
              Geometry and Topology (CoGeTo)},
    VOLUME = {1},
      YEAR = {1966},
     PAGES = {51--57},
      ISSN = {0081-6906},
   MRCLASS = {05.40},
  MRNUMBER = {205876},
MRREVIEWER = {W. Moser},
}

@article {ErSt,
    AUTHOR = {Erd{\H o}s, P. and Stone, A. H.},
     TITLE = {On the structure of linear graphs},
   JOURNAL = {Bull. Amer. Math. Soc.},
  FJOURNAL = {Bulletin of the American Mathematical Society},
    VOLUME = {52},
      YEAR = {1946},
     PAGES = {1087--1091},
      ISSN = {0002-9904},
   MRCLASS = {56.0X},
  MRNUMBER = {18807},
MRREVIEWER = {H. S. M. Coxeter},
       DOI = {10.1090/S0002-9904-1946-08715-7},
       URL = {https://doi.org/10.1090/S0002-9904-1946-08715-7},
}

@article {KMSV,
    AUTHOR = {Keevash, Peter and Mubayi, Dhruv and Sudakov, Benny and
              Verstra\"{e}te, Jacques},
     TITLE = {Rainbow {T}ur\'{a}n problems},
   JOURNAL = {Combin. Probab. Comput.},
  FJOURNAL = {Combinatorics, Probability and Computing},
    VOLUME = {16},
      YEAR = {2007},
    NUMBER = {1},
     PAGES = {109--126},
      ISSN = {0963-5483},
   MRCLASS = {05C35},
  MRNUMBER = {2286514},
MRREVIEWER = {Ivan Pashov},
       DOI = {10.1017/S0963548306007760},
       URL = {https://doi.org/10.1017/S0963548306007760},
}

@article {JoPaSa,
    AUTHOR = {Johnston, Daniel and Palmer, Cory and Sarkar, Amites},
     TITLE = {Rainbow {T}ur\'{a}n problems for paths and forests of stars},
   JOURNAL = {Electron. J. Combin.},
  FJOURNAL = {Electronic Journal of Combinatorics},
    VOLUME = {24},
      YEAR = {2017},
    NUMBER = {1},
     PAGES = {Paper No. 1.34, 15},
   MRCLASS = {05C35 (05C15)},
  MRNUMBER = {3625911},
MRREVIEWER = {Craig Michael Timmons},
       DOI = {10.37236/6430},
       URL = {https://doi.org/10.37236/6430},
}

@article {ErGyMe,
    AUTHOR = {Ergemlidze, Beka and Gy\H{o}ri, Ervin and Methuku, Abhishek},
     TITLE = {On the rainbow {T}ur\'{a}n number of paths},
   JOURNAL = {Electron. J. Combin.},
  FJOURNAL = {Electronic Journal of Combinatorics},
    VOLUME = {26},
      YEAR = {2019},
    NUMBER = {1},
     PAGES = {Paper No. 1.17, 12},
   MRCLASS = {05C35 (05C38)},
  MRNUMBER = {3919626},
}

@article{JoRo,
    AUTHOR = {Johnston, Daniel and Rombach, Puck},
     TITLE = {Lower bounds for rainbow {T}ur\'{a}n numbers of paths and other
              trees},
   JOURNAL = {Australas. J. Combin.},
  FJOURNAL = {The Australasian Journal of Combinatorics},
    VOLUME = {78},
      YEAR = {2020},
     PAGES = {61--72},
      ISSN = {1034-4942},
   MRCLASS = {05C35 (05C15)},
  MRNUMBER = {4148505},
MRREVIEWER = {Ioan Tomescu},
}

@article{halfpaprainbowp5,
      title={The rainbow {T}ur\'an number of {$P_5$}}, 
      author={Anastasia Halfpap},
      journal = {Australas. J. Combin.},
      fjournal = {The Australasian Journal of Combinatorics},
      volume = {87 (3)},
      YEAR = {2023},
      PAGES = {403--422},
      year={2022}
}

@article{Kirkman,
  title={On a problem in combinations},
  author={T.P. Kirkman},
  JOURNAL = {Cambridge and Dublin Math. J.},
  FJOURNAL = {The Cambridge and Dublin Mathematical Journal},
  year={1847},
  volume={2},
  pages={191-204}
}

@article {Plantholt,
    AUTHOR = {Plantholt, Mike},
     TITLE = {The chromatic index of graphs with a spanning star},
   JOURNAL = {J. Graph Theory},
  FJOURNAL = {Journal of Graph Theory},
    VOLUME = {5},
      YEAR = {1981},
    NUMBER = {1},
     PAGES = {45--53},
      ISSN = {0364-9024,1097-0118},
   MRCLASS = {05C15},
  MRNUMBER = {604305},
       DOI = {10.1002/jgt.3190050103},
       URL = {https://doi.org/10.1002/jgt.3190050103},
}
\end{document}